\documentclass[12pt,a4paper]{amsart}
\usepackage{a4wide}
\usepackage{amsfonts,amsthm,amsmath,amssymb,amscd}
\usepackage{graphicx,color}

\theoremstyle{plain}
\newtheorem{lem}{Lemma}
\newtheorem{prop}[lem]{Proposition}
\newtheorem{thm}[lem]{Theorem}
\newtheorem*{thm*}{Theorem}
\newtheorem{cor}[lem]{Corollary}
\newtheorem*{cor*}{Corollary}

\theoremstyle{definition}
\newtheorem{defn}[lem]{Definition}
\newtheorem*{defn*}{Definition}

\newtheorem*{ex*}{Example}
\newtheorem{rem}[lem]{Remark}
\newtheorem*{rem*}{Remark}

\theoremstyle{remark}
\newtheorem*{notat}{Notation}

\DeclareMathOperator{\id}{id}

\DeclareMathOperator{\diam}{diam}
\DeclareMathOperator{\dist}{dist}
\DeclareMathOperator{\graph}{graph}
\DeclareMathOperator{\card}{card}

\newcommand{\R}{\mathbb R}
\newcommand{\Z}{\mathbb Z}
\newcommand{\N}{\mathbb N}

\newcommand{\MM}{\mathcal M}
\newcommand{\HH}{\mathcal H}

\newcommand{\U}{\mathcal U}

\newcommand{\II}{\mathcal I}

\begin{document}

\title[On the dimension of graphs of Weierstrass-type functions]{On the dimension of graphs of Weierstrass-type functions with rapidly growing frequencies}

\date{\today}

\author{Krzysztof Bara\'nski}
\address{Institute of Mathematics, University of Warsaw,
ul.~Banacha~2, 02-097 Warszawa, Poland}
\email{baranski@mimuw.edu.pl}

\thanks{Research partially supported by Polish MNiSW Grant N N201 607940.}

\begin{abstract} We determine the Hausdorff and box dimension of the fractal graphs for a general class of Weierstrass-type functions of the form $f(x) = \sum_{n=1}^\infty a_n \, g(b_n x + \theta_n)$, where $g$ is a periodic Lipschitz real function and $a_{n+1}/a_n \to 0$, $b_{n+1}/b_n \to \infty$ as $n \to \infty$. Moreover, for any $H, B \in [1, 2]$, $H \leq B$ we provide examples of such functions with $\dim_H(\graph f) = \underline{\dim}_B(\graph f) = H$, $\overline{\dim}_B(\graph f) = B$.
\end{abstract}

\maketitle
\section{Introduction}
In this paper we study the dimension of the graphs of real functions of the form 
\begin{equation}\label{eq:f}
f:\R \to \R, \qquad f(x) = \sum_{n=1}^\infty a_n \, g(b_n x + \theta_n),
\end{equation}
where $g: \R \to \R$ is a non-constant periodic Lipschitz function, $a_n, b_n > 0$ with $b_{n+1}/b_n \to \infty$ as $n\to \infty$ and $\theta_n \in \R$. The most well-known class of functions of the form \eqref{eq:f} is the one with $b_n = b^n$, $a_n = b^{-\alpha n}$ for $b > 1$, $\alpha \in (0, 1)$, including the famous Weierstrass example of a continuous nowhere differentiable function on the interval. In spite of many efforts (see e.g. \cite{He,Hu,L,MW,PU} and the references therein), the question of determining the Hausdorff dimension of the graphs of such functions is still (mostly) open. 

It turns out that the case $b_{n+1}/b_n \to \infty$ is easier to handle. 
Probably the first to study such functions were Besicovitch and Ursell in 1937 \cite{BU}, who considered the case when $g$ is the ``sawtooth'' function $\Lambda(x) = \dist(x, \Z)$, with $a_n = b_n^{-\alpha}$ for some $\alpha \in (0, 1)$ and $\theta_n = 0$. In that case, they showed that if the sequence $b_{n+1}/b_n$ increases to $\infty$ and $\log b_{n+1}/ \log b_n \to 1$ as $n \to \infty$, then the Hausdorff dimension of the graph of $f$ is equal to $2 - \alpha$. Moreover, if $b_n = b_1^{\beta^{n-1}}$ where $b_1 > 1$ (then $\log b_{n+1}/ \log b_n \to \beta$ as $n \to \infty$) and
\[
\beta = \frac{(1-\alpha)(2-H)}{\alpha(H-1)}
\]
for $1 < H < 2-\alpha$, then the graph of $f$ has Hausdorff dimension $H$. 

In 1995, Wingren showed in \cite{W} that the graph of the function
\[
f(x) = \sum_{n=1}^\infty 2^{-n} \, \Lambda(2^{2^n} x)
\]
has Hausdorff dimension equal to two (and so has any subset of the graph whose projection on the real axis has positive Lebesgue measure). In \cite{Li} Liu showed that any such subset of the graph of the function
\[
f(x) = \sum_{n=1}^\infty 2^{-n(n-1)-1} \, \Lambda(2^{n(n+1)+1} x)
\]
has Hausdorff dimension equal to one and infinite $1$-Hausdorff measure.

Recently, Carvalho \cite{C} showed that if $g$ is a periodic Lipschitz function, such that $g$ is increasing on some interval $I_1$ and decreasing on some interval $I_2$, with $|g(x) - g(y)| > \delta |x-y|$ for every $x, y$ in $I_1$ and every $x, y$ in $I_2$, for some constant $\delta > 0$, moreover $\log b_{n+1}/ \log b_n \to \beta$  and $- \log b_n/\log a_n \to \alpha$ as $n \to \infty$ for $\alpha \in (0, 1)$, $\beta > 1$, then 
\begin{align*}
\dim_H(\graph f) = \underline{\dim}_B(\graph f) &= 1 + \frac{1-\alpha}{1-\alpha + \alpha\beta}, \\ 
\overline{\dim}_B(\graph f) &= 2 -\alpha,
\end{align*}
for the function $f$ of the form \eqref{eq:f} (for arbitrary $\theta_n$). Here $\dim_H$, $\underline{\dim}_B$ and $\overline{\dim}_B$ denote respectively the Hausdorff, lower and upper box dimension (see Section~\ref{sec:prelim} for definitions).

In this paper we complete the above results, determining the Hausdorff and box dimension of the graphs of functions of the form \eqref{eq:f} in the general case $b_{n+1}/b_n \to\infty$, $a_{n+1}/a_n \to 0$. More precisely, we prove the following.

\begin{thm}\label{thm:main} Let $g: \R \to \R$ be a periodic Lipschitz function, such that $g$ is strictly monotone on some $($non-trivial$)$ interval $I \subset \R$ with 
$|g(x) - g(y)| > \delta|x-y|$ for every $x,y \in I$ and a constant $\delta > 0$. If $a_n, b_n >0$, $a_{n+1}/a_n \to 0$, $b_{n+1}/b_n \to \infty$ as $n \to \infty$ and $\theta_n \in \R$, then for functions $f$ of the form \eqref{eq:f} we have
\begin{align*}
\dim_H(\graph f) = \underline{\dim}_B(\graph f) &= 1 + \liminf_{n\to\infty} \frac{\log^+ d_n}{\log (b_{n+1}d_n/d_{n+1})}, \\ 
\overline{\dim}_B(\graph f) &= 1 + \limsup_{n\to\infty} \frac{\log^+ d_n}{\log b_n},
\end{align*}
where
\[
d_n = a_1b_1 + \cdots + a_nb_n.
\]
\end{thm}

\begin{rem}\label{rem:cantor}
In fact, the proof shows that $\dim_H (\graph f) = \dim_H (\graph f|_\II)$ for some Cantor set $\II \subset \R$ of Lebesgue measure $0$.
\end{rem}

\begin{rem} The assumptions on the function $g$ are satisfied, if 
$g$ is a periodic Lipschitz function, which is non-constant and $C^1$
on some interval (e.g. if $g$ is a periodic non-constant $C^1$ function). Indeed, in this case there exists an interval $I$ with $g'|_I > \delta > 0$ or $g'|_I < - \delta < 0$. 
\end{rem}

\begin{rem}\label{rem:a_n}
The assertion on $\overline{\dim}_B(\graph f)$ holds under a weaker assumption: $a_{n+1}/a_n \to 0$ can be replaced by $a_{n+1} < \eta a_n$ for large $n$, where $\eta$ is a sufficiently small constant depending on $g$ $($not on the sequences $a_n, b_n, \theta_n)$.
\end{rem}

Theorem~{\rm \ref{thm:main}} implies immediately the following corollaries.

\begin{cor}\label{cor:a_n} Under the assumption of Theorem~{\rm \ref{thm:main}}, if additionally $a_{n+1}b_{n+1} \geq a_nb_n \geq 1$ for sufficiently large $n$, then
\begin{align*}
\dim_H(\graph f) = \underline{\dim}_B(\graph f) &= 1 + \liminf_{n\to\infty} \frac{\log (a_n b_n)}{\log (a_n b_n / a_{n+1})}, \\ 
\overline{\dim}_B(\graph f) &= 2 - \limsup_{n\to\infty} \frac{\log a_n}{\log b_n},
\end{align*}
\end{cor}
\begin{proof} In this case $d_n > 1$ and $a_nb_n < d_n \leq c n a_nb_n$ for some constant $c > 0$, so
\[
0 < \log^+ d_n - \log (a_nb_n)  < \log n + \log c, 
\]
which easily implies the assertion (see \eqref{eq:n/b_n}).
\end{proof}

\begin{cor}\label{cor:b_n} Under the assumption of Theorem~{\rm \ref{thm:main}}, if additionally $\log b_{n+1}/\log b_n \to 1$ as $n\to\infty$, then 
\[
\dim_H(\graph f) = \underline{\dim}_B(\graph f) = 1 + \liminf_{n\to\infty} \frac{\log^+ d_n}{\log b_n}.
\]
In this case $\dim_B(\graph f)$ exists if and only if there exists the limit $\displaystyle \gamma = \lim_{n\to\infty}\log^+ d_n/\log b_n$ and then 
\[
\dim_H(\graph f) = \dim_B(\graph f) = 1 + \gamma.
\]
\end{cor}
\begin{proof}
It is sufficient to show
\[
\lim_{n\to\infty} \frac{\log (b_{n+1}d_n / d_{n+1})}{\log b_n} = 1,
\]
which is equivalent to 
\begin{equation}\label{eq:to0}
\lim_{n\to\infty} \frac{\log (d_{n+1} / d_n)}{\log b_n} = 0.
\end{equation}
Since $a_{n+1}/a_n \to 0$, $b_{n+1}/b_n \to \infty$ as $n \to \infty$, we have $a_{n+1}<a_n$ and $b_{n+1}>b_n$ for large $n$, so
\[
0 < \log \frac{d_{n+1}}{d_n} = \log\left(1 + \frac{a_{n+1}b_{n+1}}{a_1b_1+\cdots + a_nb_n}\right) <  \log\left(1 + \frac{b_{n+1}}{b_n}\right) < \log b_{n+1} -\log b_n + \log 2,
\]
which gives \eqref{eq:to0}.
\end{proof}

\begin{cor}\label{cor:alpha}
Let $g$ be like in Theorem~{\rm \ref{thm:main}} and let
\[
f(x) = \sum_{n=1}^\infty b_n^{-\alpha} \, g(b_n x + \theta_n),
\]
where $\alpha \in (0, 1]$, $b_n > 0$, $\theta_n \in \R$, and $b_{n+1}/b_n \to \infty$ as $n \to \infty$. Then 
\begin{align*}
\dim_H(\graph f) = \underline{\dim}_B(\graph f) &= 1 + \frac{1-\alpha}{\displaystyle 1 -\alpha + \alpha \limsup_{n\to\infty} \log b_{n+1}/\log b_n},\\
\overline{\dim}_B(\graph f) &= 2 - \alpha.
\end{align*}
$($This includes the case $\limsup_{n\to\infty} \log{b_{n+1}/\log b_n}=\infty$ with the convention $1/\infty = 0$.$)$
In particular, if $\log b_{n+1}/\log b_n \to 1$ as $n\to\infty$, then
\[
\dim_H(\graph f) = \dim_B(\graph f)= 2 - \alpha.
\]
\end{cor}
\begin{proof} Let $a_n = b_n^{-\alpha}$. Then $d_n = a_1b_1 + \cdots + a_nb_n = b_1^{1-\alpha}+\cdots + b_n^{1-\alpha}$. 
Since $b_{n+1}/b_n \to \infty$, we have
\[
b_n^{1-\alpha} < d_n < 2 b_n^{1-\alpha}
\]
for large $n$, which gives
\[
\lim_{n\to\infty} \frac{\log^+ d_n}{\log b_n} = 1 -\alpha
\]
and
\begin{align*}
\liminf_{n\to\infty} \frac{\log^+ d_n}{\log (b_{n+1} d_n / d_{n+1})} &= \liminf_{n\to\infty}\frac{(1-\alpha)\log b_n}{\alpha\log b_{n+1} +(1-\alpha)\log b_n}\\ &= \frac{1-\alpha}{\displaystyle 1 -\alpha + \alpha \limsup_{n\to\infty} \log b_{n+1}/\log b_n}.
\end{align*}
\end{proof}

It is obvious that $\dim_H(\graph f), \overline{\dim}_B(\graph f) \in [1,2]$ and $\dim_H(\graph f)\leq \overline{\dim}_B(\graph f)$ (see Section~\ref{sec:prelim}). Using Theorem~{\rm \ref{thm:main}}, we can provide examples of a function $f$ of the form $\eqref{eq:f}$ with any prescribed values of $\dim_H(\graph f)$ and $\overline{\dim}_B(\graph f)$.

\begin{cor}\label{cor:ex} For every $H, B \in [1, 2]$, such that $H \leq B$ there exists a function $f$ fulfilling the assumptions of Theorem~{\rm\ref{thm:main}}, such that 
\[
\dim_H(\graph f) = \underline{\dim}_B(\graph f) =H, \qquad 
\overline{\dim}_B(\graph f) = B.
\]
\end{cor}
\begin{proof} A large part of examples is not new --- we present them for completeness. The function $g$ can be taken to be e.g. $\sin$, $\cos$, $\dist(\cdot, \Z)$ etc.

For $H = B \in [1, 2)$, it is enough to take 
\[
f(x) = \sum_{n=1}^\infty n^{-(2-B) n} g (n^nx),
\]
and use Corollary~\ref{cor:alpha} for $\alpha = 2-B$, $b_n = n^n$. Similarly, for
$H \in (1,2)$, $B \in (H,2)$ we take 
\[
f(x) = \sum_{n=1}^\infty 2^{-(2-B)\beta^n} g (2^{\beta^n}x), 
\]
where
\[
\beta = \frac{(2-H)(B-1)}{(H-1)(2-B)}
\]
and use Corollary~\ref{cor:alpha} for $\alpha = 2-B$, $b_n = 2^{\beta^n}$. For $H = 1$, $B \in (1, 2)$ we take 
\[
f(x) = \sum_{n=1}^\infty 2^{-(2-B)n^n} g (2^{n^n}x)
\]
and use Corollary~\ref{cor:alpha} for $\alpha = 2-B$, $b_n = 2^{n^n}$. 

For $H = 1$, $B = 2$, take
\[
f(x) = \sum_{n=1}^\infty 2^{-n^n/\sqrt{n}} g (2^{n^n}x),
\]
for $H \in (1, 2)$, $B = 2$, 
\[
f(x) = \sum_{n=1}^\infty 2^{-((2-H)/(e(H-1)))n^{n-1}} g (2^{n^n}x),
\]
and for $H = B = 2$, 
\[
f(x) = \sum_{n=1}^\infty n^{-\sqrt{n}} g (n^nx).
\]
In all three cases, we have $a_{n+1}b_{n+1}/(a_nb_n) \to \infty$ as $n\to\infty$, so we can use Corollary~\ref{cor:a_n}. Easy details are left to the reader.
\end{proof}

The plan of the paper is as follows. After preliminaries, in Section~\ref{sec:lem}, we prove some useful lemmas. The estimates for the box dimension are done in Section~\ref{sec:BD} (Propositions~\ref{prop:upperBD}--\ref{prop:lowerBD}) and the Hausdorff dimension is estimated in Section~\ref{sec:HD} (Proposition~\ref{prop:HD} and Corollary~\ref{cor:HD}). The proof follows the construction used by Mauldin and Williams in \cite{MW}.

\section{Preliminaries} \label{sec:prelim}
For convenience, we recall the definitions of the Hausdorff and box dimension. For details, see e.g. \cite{F,M}.

\begin{defn} For $A \subset \R^n$ and $s > 0$ the (outer)
$s$-Hausdorff measure of $A$ is defined as 
\[
\HH^s (A) = \lim_{r \to 0^+} \inf \sum_{U \in \U} (\diam U)^s,
\]
where infimum is taken over all countable coverings $\U$ of $A$ by
open sets of diameters smaller than $r$. The Hausdorff dimension of $A$ is defined as
\[
\dim_H (A) = \sup \{s > 0 : \HH^s (A) = + \infty\} = \inf
\{s > 0 : \HH^s (A) = 0\}.
\]
\end{defn}

\begin{defn} For a bounded set $A \subset \R^n$ and $r > 0$ let $N(r)$ be the minimal number of balls of diameter $r$ needed to cover $A$. The lower and upper box
dimension (also called the box-counting or Minkowski dimension) are defined respectively as
\[
\underline{\dim}_B (A) = \liminf_{r \to 0^+}
\frac{\log N(r)}{-\log r}, \qquad 
\overline{\dim}_B (A) = \limsup_{r \to 0^+}
\frac{\log N(r)}{-\log r}.
\]
The lower and upper box dimension dimension of an arbitrary set $A \subset \R^n$ are defined as 
\[
\underline{\dim}_B (A) = \sup_{X \text{ bounded}} \underline{\dim}_B(A \cap X), \qquad \overline{\dim}_B (A) = \sup_{X \text{ bounded}} \overline{\dim}_B(A \cap X).
\]
If $\underline{\dim}_B (A) = \overline{\dim}_B (A)$, then the common value is called the box dimension of $A$ and is denoted by $\dim_B(A)$. 
\end{defn}

We have
\[
\dim_H(A) \leq \underline{\dim}_B (A).
\]

The definitions of the Hausdorff and box dimension easily imply

\begin{lem} \label{lem:holder} Let $f: [a, b] \to \R$ for $a, b \in\R$, $a<b$. Then
$\dim_H(\graph f) \geq 1$. Moreover, if $f$ is is H\"older continuous with exponent $\alpha \in (0,1]$, i.e. 
\[
\| f(x) - f(y) \| \leq c \| x - y \|^\alpha
\]
for every $x, y \in A$ and some constant $c > 0$, then
\[
\overline{\dim}_B(\graph f) \leq 2 - \alpha.
\]
In particular, if $\alpha = 1$, i.e. $f$ is Lipschitz continuous, then 
\[
\dim_H(\graph f) = \dim_B(\graph f) = 1.
\]
\end{lem}

\begin{defn}
The Hausdorff dimension of a finite (non-zero) Borel measure $\nu$ in $\R^n$ is defined as
\[                                                                              
\dim_H(\nu) = \inf\{\dim_H(A): A \text{ has full measure } \mu\}.                     
\]      
\end{defn}
                           
A well-known tool for estimating the Hausdorff dimension is the following fact (see e.g. \cite{F,M}).
                                                                                
\begin{lem}\label{lem:frostman} Let $B_r(x)$ denote the ball in $\R^n$ centred at $x$ of radius $r$. If 
\[
\liminf_{r \to 0^+}\frac{\log \nu (B_r(x))}{\log r} \geq D
\]
for $\nu$-almost every $x$, then $\dim_H(\nu) \geq D$, in particular $\dim_H(A) \geq D$ for every Borel set $A$ of positive measure $\nu$. The ball $B_r(x)$ can be replaced by an $n$-dimensional cube centred at $x$ of side $r$.
\end{lem}          

\begin{notat}
We set $\N = \{1, 2, 3, \ldots\}$. By $\card$ we denote the cardinality of a set. The integer part of $x \in \R$ (i.e. the largest integer not larger than $x$) is denoted by $[x]$. The $1$-dimensional Lebesgue measure of a set $A \subset \R$ is denoted by $|A|$.
\end{notat}

\section{Lemmas}\label{sec:lem}
Let $f$ satisfy the assumptions of Theorem~\ref{thm:main}. 
Obviously, we can assume that the period of $g$ is equal to $1$, the interval $I$ is closed and contained in $[0, 1]$ and $g$ is strictly increasing on $I$. Since $g$ is Lipschitz and
\[
\dim_H(\graph(f+h)) = \dim_H(\graph f), \quad \dim_B(\graph(f+h)) = \dim_B(\graph f)
\]
for every Lipschitz function $h:\R \to \R$ (see e.g. \cite{MW}), 
taking $h = \sum_{n=1}^N a_n \, g(b_n x + \theta_n)$ for any $N \geq 1$, we can replace $f$ by $\sum_{n=N+1}^\infty a_n \, g(b_n x + \theta_n)$. Therefore, we can assume
\begin{equation}\label{eq:eta}
\frac{a_{n+1}}{a_n} < \eta, \qquad\frac{b_{n+1}}{b_n} > \frac{1}{\eta}
\end{equation}
for an arbitrarily small fixed $\eta > 0$ (cf. Remark~\ref{rem:a_n}).
Note also that by the Stolz-Ces\'aro theorem,
\begin{equation}\label{eq:n/b_n}
\lim_{n\to\infty} \frac{n}{\log b_n} = \lim_{n\to\infty} \log\frac{b_n}{b_{n+1}} = 0.
\end{equation}

\begin{lem}\label{lem:d/b}
We have 
\[
\frac{d_{n+1}}{b_{n+1}} < 2\eta \frac{d_n}{b_n}
\]
for $\eta$ from \eqref{eq:eta}. In particular, $d_n/b_n \to 0$ as $n\to\infty$.
\end{lem}
\begin{proof}
By \eqref{eq:eta}, we have
\[
\frac{d_{n+1}b_n}{d_nb_{n+1}} = \left(1+\frac{a_{n+1}b_{n+1}}{a_1b_1+\cdots +a_nb_n}\right)\frac{b_n}{b_{n+1}} < \frac{b_n}{b_{n+1}} + \frac{a_{n+1}}{a_n} < 2\eta < 1,
\]
if $\eta$ is chosen sufficiently small.
\end{proof}

\begin{lem}\label{lem:||<} There exists a constant $c_0 > 0$, such that for every $x, y \in \R$ and every $n\in\N$,
\[
|f(x)-f(y)| \leq c_0 (d_n|x-y| + a_{n+1}).
\]
\end{lem}
\begin{proof} Let $f_i(t) = a_i g(b_i t+\theta_i)$ for $i \in \N$. Then
\[
|f_i(x)-f_i(y)| \leq L a_i b_i |x-y|,
\]
where $L$ is the Lipschitz constant of $g$. This together with \eqref{eq:eta} implies
\begin{multline*}
|f(x)-f(y)| \leq \sum_{i=1}^n |f_i(x)-f_i(y)| + 2 \sup\left|f - \sum_{i=1}^n f_i\right|\\ \leq L d_n |x-y| + 2\sup |g| \sum_{i=n+1}^\infty a_i < L d_n |x-y| + \frac{2\sup |g|}{1-\eta}\; a_{n+1} <  c_0 (d_n|x-y| + a_{n+1})
\end{multline*}
for a suitable constant $c_0 > 0$.
\end{proof}

Let 
\[
I_{0,0} = I, \qquad I_{n, j} = \frac{I - \theta_n + j}{b_n}
\]
for $n \in \N$, $j \in \mathcal J_n$, where $\mathcal J_n \subset \Z$ is defined inductively as
\[
\mathcal J_0 = \{0\}, \qquad \mathcal J_n = \{j \in \Z: I_{n,j} \subset I_{n-1,j'} \text{ for some } j' \in \mathcal J_{n-1}\}
\]
for $n \in \N$. Then every $I_{n, j}$ is an interval of length $|I|/b_n$ and the gap between two consecutive intervals $I_{n,j}, I_{n,j+1} \subset I_{n-1,j'}$ has length $(1-|I|)/b_n$. By definition,
\begin{equation}\label{eq:I_n,j}
b_n x+\theta_n\in I \mod 1 \qquad \text{for every } x \in I_{n,j}.
\end{equation}
We will call $I_{n,j}$ intervals of $n$-th generation. By \eqref{eq:eta}, we can assume that every interval of $n$-th generation contains at least two intervals of next generation. 

\begin{lem}\label{lem:J_n}
There exists $q > 0$ such that the interval $I$ contains more than $q b_1$ intervals of first generation and for every $n\in\N$, every interval of $n$-th generation contains more than $q b_{n+1}/b_n$ intervals of $(n+1)$-th generation. Moreover,
\[
\card \mathcal J_n > q^n b_n.
\]
\end{lem}
\begin{proof}
For $n \geq 0$, $j\in\mathcal J_n$ let $N_{n,j}$ be the number of intervals of $(n+1)$-th generation contained in $I_{n,j}$. Since the intervals of $(n+1)$-th generation have length $|I|/b_{n+1}$ and are separated by gaps of length at least $(1-|I|)/b_{n+1}$, we have, setting $b_0 = 1$,
\[
\frac{|I|}{b_n} < \frac{N_{n,j}+2}{b_{n+1}},
\]
so by \eqref{eq:eta},
\[
N_{n,j} > |I|\frac{b_{n+1}}{b_n} - 2 > q \frac{b_{n+1}}{b_n}
\]
for some constant $q > 0$. Moreover,
\[
\card \mathcal J_n \geq N_{0,0}\min_{j_1\in \mathcal J_1} N_{1,j_1} \cdots \min_{j_{n-1}\in \mathcal J_{n-1}} N_{n-1,j_{n-1}} > q^n b_n.
\]
\end{proof}
\begin{lem}\label{lem:||>} There exist $c_1, c_2 > 0$, such that for every $n \in \N$, $j \in \mathcal J_n$ and every $x, y \in I_{n,j}$,
\[
|f(x)-f(y)| \geq c_1d_n|x-y| - c_2a_{n+1}.
\]
\end{lem}
\begin{proof} We can assume $x > y$. Since every interval of $n$-th generation is contained in an interval of $i$-th generation for every $i \leq n$, by \eqref{eq:I_n,j} we have 
\[
f_i(x) - f_i(y) > \delta a_i b_i (x-y) \qquad \text{for } i = 1, \ldots, n, 
\]
so
\begin{multline*}
f(x)-f(y) > \delta d_n(x-y) - 2 \sup\left|f - \sum_{i=1}^n f_i\right| \geq \delta d_n(x-y) - 2\sup |g| \sum_{i=n+1}^\infty a_i \\> \delta d_n(x-y) - \frac{2\sup |g|}{1-\eta}\; a_{n+1} > c_1d_n|x-y| - c_2a_{n+1}
\end{multline*}
for suitable constants $c_1, c_2 > 0$.
\end{proof}

\section{Box dimension}\label{sec:BD}

Now we prove (simultaneously) two following propositions.

\begin{prop}\label{prop:holder} Let $\overline{\gamma} = \limsup_{n\to\infty} \log^+ d_n/\log b_n$. Then the following assertions hold:
\begin{itemize}
\item[\rm (a)] 
The function $f$ is H\"older continuous with every exponent smaller than 
$1 - \overline{\gamma}$ and is not H\"older continuous with any exponent larger than $1 - \overline{\gamma}$.
\item[\rm (b)] If $a_n = b_n^{-\alpha}$ for some $\alpha \in (0, 1]$, then $f$ is H\"older continuous with exponent $\alpha = 1 - \overline{\gamma}$.
\item[\rm (c)] If the sequence $d_n$ is bounded, then $f$ is Lipschitz continuous, so $\dim_H(\graph f) = \dim_B(\graph f) = 1$.
\end{itemize}
\end{prop}

\begin{prop}\label{prop:upperBD}
We have
\[
\overline{\dim}_B(\graph f) = 1 + \limsup_{n\to\infty} \frac{\log^+ d_n}{\log b_n}.
\]
\end{prop}

\begin{proof}
For a set $A \subset \R$ let
\[
V_A = \sup_A f - \inf_A f. 
\]
Take a small $r > 0$. Let $k = k(r) \in \N$ be such that
\begin{equation}\label{eq:k}
\frac{1}{b_{k+1}} \leq  r < \frac{1}{b_k}. 
\end{equation}
By Lemma~\ref{lem:||<}, for every $t\in \R$ we have
\begin{equation}\label{eq:V<}
V_{[t, t+r]} \leq c_0(d_k r + a_{k+1}) < c_0\left(d_k r + \frac{d_{k+1}}{b_{k+1}} \right) \leq 2 c_0 \max\left(d_k r,\, \frac{d_{k+1}}{b_{k+1}}\right).
\end{equation}
In particular, if the sequence $d_n$ is bounded, then \eqref{eq:k} and \eqref{eq:V<} give 
\[
V_{[t, t+r]} < 2c_0\max\left(d r, \,\frac{d}{b_{k+1}} \right) \leq 2c_0 d r
\]
for some constant $d > 0$, which means that $f$ is Lipschitz and proves the assertion~(c) of Proposition~\ref{prop:holder}. Hence, we assume from now on that $d_n \to \infty$ as $n\to\infty$, in particular $d_n > 1$ and $\log^+ d_n = \log d_n$ for large $n$. 

By \eqref{eq:k} and Lemma~\ref{lem:d/b}, we have
\begin{align*}
d_kr &= r^{1 + \log d_k/\log r} < r^{1 - \log d_k/\log b_k},\\
\frac{d_{k+1}}{b_{k+1}} &= \left(\frac{1}{b_{k+1}}\right)^{1 - \log d_{k+1}/\log b_{k+1}}
\leq r^{1 - \log d_{k+1}/\log b_{k+1}}.
\end{align*}

Hence, \eqref{eq:V<} implies
\begin{equation}\label{eq:V_r<}
V_{[t, t+r]} \leq 2c_0 r^{1 - \max(\log d_k/\log b_k, \, \log d_{k+1}/\log b_{k+1})}.
\end{equation}

On the other hand, by Lemma~\ref{lem:||>}, \eqref{eq:eta} and Lemma~\ref{lem:d/b}, for every $n\in\N$, $j\in \mathcal J_n$ we have
\begin{multline}\label{eq:V_I>}
V_{I_{n,j}} \geq c_1 |I|\frac{d_n}{b_n} - c_2a_{n+1} > c_1|I|\frac{d_n}{b_n} - c_2\eta a_n 
> (c_1|I|-\eta c_2)\frac{d_n}{b_n}\\ = (c_1|I|-\eta c_2) \left(\frac{1}{b_n}\right)^{1- \log d_n /\log b_n} > (c_1|I|-\eta c_2) |I_{n,j}|^{1- \log d_n /\log b_n},
\end{multline}
where $c_1|I|-\eta c_2 > 0$, if $\eta$ was chosen sufficiently small.

Let $s < 1 - \overline{\gamma}$. Then for sufficiently large $k$ we have 
\[
1 - \max(\log d_k/\log b_k, \, \log d_{k+1}/\log b_{k+1}) > s,                                                                                                                                                              \]
so by \eqref{eq:V_r<}, the function $f$ is H\"older continuous with exponent $s$. 

Take now $s > 1 - \overline{\gamma}$ and suppose $f$ is H\"older continuous with exponent $s$. Then there exists $\varepsilon > 0$ such that $1 - \log d_n / \log b_n < s -\varepsilon$ for infinitely many $n$, which contradicts \eqref{eq:V_I>}. In this way we have proved the assertion~(a) of Proposition~\ref{prop:holder}. Note that the assertion~(a) implies immediately that 
\begin{equation}\label{eq:BD<}
\overline\dim_B(\graph f) \leq 1 + \overline{\gamma}
\end{equation}
(see Lemma~\ref{lem:holder}).
If $a_n = b_n^{-\alpha}$ for some $\alpha \in (0, 1]$, then $d_n = b_1^{1-\alpha}+\cdots + b_n^{1-\alpha}$, so (since $b_{n+1}/b_n \to \infty$), we have 
\[
b_n^{1-\alpha} < d_n < 2 b_n^{1-\alpha}
\]
for large $n$. Hence, $\alpha = 1 - \overline{\gamma}$ and, by \eqref{eq:k} and \eqref{eq:V<}, 
\[
V_{[t,t+r]} < 4 c_0 \max\left(b_k^{1-\alpha}r, \, b_{k+1}^{-\alpha}\right) \leq 4 c_0 r^\alpha,
\]
which proves the assertion~(b) of Proposition~\ref{prop:holder}.

Let $N(r)$ be the minimal number of squares with vertical and horizontal sides of length $r$ needed to cover $\graph f|_I$. 
Since, by Lemma~\ref{lem:J_n}, for every $n$ there are more than $q^n b_n$ disjoint intervals $I_{n,j}$ of lengths $|I|/b_n$ contained in $|I|$, using \eqref{eq:V_I>} we get 
\[
N(|I_{n,j}|) > c q^n b_n |I_{n,j}|^{-\log d_n /\log b_n}
\]
for some constant $c > 0$, so (using \eqref{eq:n/b_n}) we have
\begin{multline*}
\overline\dim_B(\graph f) \geq \limsup_{n\to\infty} \frac{\log N(|I_{n,j}|)}{-\log |I_{n,j}|} \\\geq  \limsup_{n\to\infty} \frac{n\log q + \log b_n + \log d_n(1-\log|I|/\log b_n)}{\log b_n} = 1 + \overline{\gamma}.
\end{multline*}
This together with \eqref{eq:BD<} gives 
\[
\overline\dim_B(\graph f) = 1 +\overline{\gamma},
\]
which proves Proposition~\ref{prop:upperBD}.
\end{proof}

\begin{prop}\label{prop:lowerBD}
We have
\[
\underline{\dim}_B(\graph f) \leq 1 + \liminf_{n\to\infty} \frac{\log^+ d_n}{\log (b_{n+1}d_n/d_{n+1})}.
\]
\end{prop}
\begin{proof}Take a small $r > 0$. Let $k = k(r), m  = m(r) \in \N$ be such that
\begin{equation}\label{eq:k,m}
\frac{1}{b_{k+1}} \leq  r < \frac{1}{b_k}, \qquad \frac{m}{b_{k+1}} \leq r < \frac{m+1}{b_{k+1}}. 
\end{equation}
Obviously, $k \to \infty$ as $r \to 0^+$. By definition, $ m \in \{1, \ldots, m_k\}$, where
\[
m_k = \begin{cases} b_{k+1}/b_k - 1 & \text{if } b_{k+1}/b_k \in \N\\
[b_{k+1}/b_k] & \text{otherwise}.  
\end{cases}
\]
By Proposition~\ref{prop:holder}~(c), we can assume that $d_k > 1$.
By \eqref{eq:V<}, for every $t\in\R$
\[
V_{[t,t+r]} < 2c_0\max\left(\frac{(m+1)d_k}{b_{k+1}}, \, \frac{d_{k+1}}{b_{k+1}}\right) \leq 4c_0 \max\left(\frac{md_k}{b_{k+1}}, \, \frac{d_{k+1}}{b_{k+1}}\right),
\]
so by \eqref{eq:k,m},
\begin{align*}
\underline{\dim}_B(\graph f) &\leq 2 + \liminf_{r\to 0^+}\sup_{t\in\R}\frac{\log V_{[t,t+r]}}{-\log r}\\ &\leq 2 + \liminf_{r\to 0^+}\max\left(\frac{\log (m d_k / b_{k+1})}{\log (b_{k+1}/m)}, \, \frac{\log (d_{k+1}/b_{k+1})}{\log (b_{k+1}/m)}\right)\\ &= 1 + \liminf_{r\to 0^+}\max\left(\frac{\log d_k}{\log (b_{k+1}/m)}, \, \frac{\log (d_{k+1}/m)}{\log (b_{k+1}/m)}\right).
\end{align*}
To end the proof of the proposition, it remains to use Lemma~\ref{lem:liminf}, which we prove below.
\end{proof}

\begin{lem}\label{lem:liminf}
For $r > 0$ and $k = k(r)$, $m= m(r)$ defined in \eqref{eq:k,m}, we have
\[
\liminf_{r\to 0^+}\max\left(\frac{\log d_k}{\log (b_{k+1}/m)},\, \frac{\log (d_{k+1}/m)}{\log (b_{k+1}/m)}\right) = \liminf_{n\to\infty} \frac{\log d_n}{\log (b_{n+1}d_n/d_{n+1})}.
\]
\end{lem}

\begin{proof}
For $r$ such that $k=k(r)$ is constant, (i.e. $r \in [1/b_{k+1}, 1/b_k)$), let
\[
X_k(m) = X_k(m(r)) = \frac{\log d_k}{\log (b_{k+1}/m)}, \qquad Y_k(m) = Y_k(m(r)) =  \frac{\log (d_{k+1}/m)}{\log (b_{k+1}/m)}
\]
for $m = m(r) \in \{1, \ldots, m_k\}$ and let
\[
M_k = \min_{m \in \{1, \ldots, m_k\}} \max(X_k(m),\, Y_k(m)).
\]
It is clear that 
\begin{equation}\label{eq:clear}
\liminf_{r\to 0^+} \max(X_k(m),\, Y_k(m)) = \liminf_{r\to 0^+} M_k.
\end{equation}
Note that $X_k$ is an increasing function of $m$, while $Y_k$ is a decreasing one (since, by Lemma~\ref{lem:d/b}, we have $d_{k+1} < b_{k+1}$). Moreover, 
\[
X_k(1) = \frac{\log d_k}{\log b_{k+1}} < \frac{\log d_{k+1}}{\log b_{k+1}} = Y_k(1)
\]
and
\[
X_k(m) \leq Y_k(m) \iff m \leq \frac{d_{k+1}}{d_k}. 
\]
By Lemma~\ref{lem:d/b} and \eqref{eq:eta}, we have
\[
1 < \frac{d_{k+1}}{d_k} < 2\eta \frac{b_{k+1}}{b_k} < \frac{b_{k+1}}{b_k} - 2 < m_k,
\]
so
\[
\left[\frac{d_{k+1}}{d_k}\right] \geq 1, \qquad 
\left[\frac{d_{k+1}}{d_k}\right] + 1 \leq m_k.
\]
Hence,
\begin{equation}\label{eq:M_k=}
M_k = \min\left(Y_k\left(\left[\frac{d_{k+1}}{d_k}\right]\right),\, X_k\left(\left[\frac{d_{k+1}}{d_k}\right]+1\right)\right).
\end{equation}

Consider the condition
\begin{equation}\label{eq:>2}
d_{k+1} \geq 2 d_k. 
\end{equation}
If \eqref{eq:>2} is not satisfied, then $[d_{k+1}/d_k] = 1$ and \eqref{eq:M_k=} gives
\begin{equation}\label{eq:Y(1)}
M_k = \min(Y_k(1), X_k(2)) = \min\left(\frac{\log d_{k+1}}{\log b_{k+1}}, \, \frac{\log d_k}{\log (b_{k+1}/2)}\right) = \frac{\log d_{k+1}}{\log b_{k+1}} + o(1)
\end{equation}
as $r \to 0^+$.
If \eqref{eq:>2} is satisfied, then 
\[
\frac{d_{k+1}}{2d_k} < \left[\frac{d_{k+1}}{d_k}\right] < \left[\frac{d_{k+1}}{d_k}\right] + 1  < \frac{2d_{k+1}}{d_k},
\]
so
\begin{equation}\label{eq:Y}
\frac{\log d_k}{\log (b_{k+1}d_k/d_{k+1})} \leq  Y_k\left(\left[\frac{d_{k+1}}{d_k}\right]\right) < \frac{\log (2d_k)}{\log (2b_{k+1}d_k/d_{k+1})} 
\end{equation}
and
\begin{equation}\label{eq:X}
\frac{\log d_k}{\log (b_{k+1}d_k/d_{k+1})} < X_k\left(\left[\frac{d_{k+1}}{d_k}\right]+1\right) <
\frac{\log d_k}{\log (b_{k+1}d_k/(2d_{k+1}))}. 
\end{equation}

Suppose first that \eqref{eq:>2} does not hold for almost all $k$. Then for all $k$ we have $\log d_k \leq k \log 2 + c$ for some constant $c > 0$, so by \eqref{eq:n/b_n}, both sides of the equation in the lemma are equal to $0$. Hence, we can assume that \eqref{eq:>2} holds for infinitely many $k$. If it holds for almost all $k$, then \eqref{eq:M_k=}, \eqref{eq:Y} and \eqref{eq:X} give
\[
\liminf_{r\to 0^+} M_k = \liminf_{r\to 0^+}\frac{\log d_k}{\log (b_{k+1}d_k/d_{k+1})} = \liminf_{n\to\infty}\frac{\log d_n}{\log (b_{n+1}d_n/d_{n+1})},
\]
which (together with \eqref{eq:clear}) ends the proof in this case. Otherwise, \eqref{eq:>2} does not hold for infinitely many $k$, and then \eqref{eq:M_k=}, \eqref{eq:Y(1)}, \eqref{eq:Y} and \eqref{eq:X} give
\[
\liminf_{r\to 0^+} M_k = \min \left(\liminf_{r\to 0^+}\frac{\log d_k}{\log b_k}, \, \liminf_{r\to 0^+}\frac{\log d_k}{\log (b_{k+1}d_k/d_{k+1})}\right).
\]
By Lemma~\ref{lem:d/b}, 
\[
\frac{\log d_k}{\log b_k} > \frac{\log d_k}{\log (b_{k+1}d_k/d_{k+1})},
\]
so 
\[
\liminf_{r\to 0^+} \frac{\log d_k}{\log b_k} \geq 
\liminf_{r\to 0^+}\frac{\log d_k}{\log (b_{k+1}d_k/d_{k+1})}.
\]
Hence, 
\[
\liminf_{r\to 0^+} M_k = \liminf_{r\to 0^+}\frac{\log d_k}{\log (b_{k+1}d_k/d_{k+1})}= \liminf_{n\to\infty}\frac{\log d_n}{\log (b_{n+1}d_n/d_{n+1})},
\]
which (together with \eqref{eq:clear}) ends the proof.
\end{proof}

Note that in this section, instead of the assumption $a_{n+1}/a_n \to 0$, we used a weaker condition \eqref{eq:eta}. This proves Remark~\ref{rem:a_n}. 

\section{Hausdorff dimension}\label{sec:HD}

Let $\II$ be the Cantor set defined as
\[
\II = \bigcap_{n=0}^\infty \bigcup_{j \in \mathcal J_n} I_{n,j}.
\]
It is obvious by construction that $\II$ has Lebesgue measure $0$ (cf. Remark~\ref{rem:cantor}). Let $\mu$ be the probabilistic Borel measure in $\R$ supported on $\II$, such that for every $I_{n+1,j'} \subset I_{n,j}$ with $j \in \mathcal J_n$, $j' \in \mathcal J_{n+1}$, we have
\[
\mu(I_{n+1,j'}) = \frac{\mu(I_{n,j})}{\card \{j''\in \mathcal J_{n+1}: I_{n+1, j''} \subset I_{n,j}\}}
\]
(the construction of such a measure is standard). By Lemma~\ref{lem:J_n}, we have
\begin{equation}\label{eq:mu(I)}
\mu(I_{n,j}) < \frac{1}{q^n b_n}
\end{equation}
for every $n\in\N$, $j \in \mathcal J_n$. Let $\nu$ be a probabilistic Borel measure in $\R^2$ supported on $\graph f|_\II$ defined by 
\[
\nu = (\id_{\R}, f)_*\mu.
\]
By definition, we have
\[
\nu(U) = \mu(\{x\in\R: (x,f(x)) \in U\})
\]
for every Borel set $U \subset \R^2$.

Recall that by Proposition~\ref{prop:holder}~(c), we assume that $d_n \to \infty$ as $n \to \infty$, in particular $d_n > 1$ and $\log^+ d_n = \log d_n$.

For $t\in \II$ and a small $r > 0$, let $Q_r(t)$ be the square with horizontal and vertical sides of length $r$ centred at $(t,f(t))\in \graph f$.

\begin{prop}\label{prop:HD}
For every $t \in \II$, we have
\[
\liminf_{r\to 0^+} \frac{\log \nu(Q_r(t))}{\log r} \geq 1 + \liminf_{n\to\infty}\frac{\log d_n}{\log (b_n d_n / d_{n+1})}.
\]
\end{prop}

Before proving Proposition~\ref{prop:HD}, we state two lemmas. 
As previously, take $k = k(r), m  = m(r) \in \N$ such that \eqref{eq:k,m} is satisfied.
Let $l = l(r) \in \N$ be such that
\begin{equation}\label{eq:l}
a_{l+1} \leq r < a_l.
\end{equation}

\begin{lem}\label{lem:l/n}
We have
\[
\lim_{r\to 0^+} \frac{l(r)}{-\log r} = 0.
\]
\end{lem}
\begin{proof} If $l \leq k$, then by \eqref{eq:n/b_n} and \eqref{eq:k,m}, we have $l/(-\log r) = o(1)$ as $r\to 0^+$. Suppose now $l > k$. Let
\[
A_n = \inf_{i\geq n} \frac{a_i}{a_{i+1}}.
\]
Since $a_{n+1}/a_n \to 0$, we have $A_{n+1} \geq A_n$ and $A_n \to \infty$.
By definition, 
\[
r < a_l = \frac{a_l}{a_{l-1}} \cdots \frac{a_{k+1}}{a_k} a_k \leq A_k^{-(l-k)} a_k < A_k^{-(l-k)}.
\]
This together with \eqref{eq:n/b_n} implies
\[
0 < \frac{l}{-\log r} = 
\frac{l-k}{-\log r} +
\frac{k}{-\log r} < \frac{1}{\log A_k} + \frac{k}{\log b_k}\to 0 
\]
as $r \to 0^+$. 
\end{proof}

Let
\[
\MM_n = \{j \in \mathcal J_n: \graph f|_{I_{n,j}} \cap Q_r(t) \neq \emptyset\}
\]
for $n \in \N$. Since the intervals of $n$-th generation are separated by gaps of length at least $(1-|I|)/b_n$, by \eqref{eq:k,m} we have
\begin{equation}\label{eq:M_k}
\card\MM_k \leq 2, \qquad \card\MM_{k+1} \leq m + 2.
\end{equation}

\begin{lem}\label{lem:M/M}
There exists a constant $c_3 > 0$ such that for every $t \in \II$, $r > 0$ and $n\in \N$, we have $\card\MM_n \neq 0$ and
\[
\frac{\card\MM_{n+1}}{\card\MM_n} \leq 
\begin{cases}
c_3d_{n+1}/d_n & \text{if } n < l,\\
c_3 & \text{if } n \geq l \text{ and } d_n/b_{n+1} > r,\\
c_3r b_{n+1}/d_n & \text{if } n \geq l \text{ and } d_n/b_{n+1} \leq r.
\end{cases}
\]
\end{lem}
\begin{proof} Since $t \in \II$, we have $\card\MM_n \neq 0$ for every $n \in \N$. Moreover,  
\begin{equation}\label{eq:M(j)}
\frac{\card\MM_{n+1}}{\card\MM_n} \leq \max_{j\in \mathcal J_n}\card \MM_{n+1}(j),
\end{equation}
where
\[
\MM_{n+1}(j) = \{j' \in \MM_{n+1}: I_{n+1,j'} \subset I_{n,j}\}.
\]
Take $j \in \mathcal J_n$ and let
\begin{align*}
x &= \max\{u \in I_{n,j}: (u, f(u)) \in Q_r(t)\},\\
y &= \min\{u \in I_{n,j}: (u, f(u)) \in Q_r(t)\}.
\end{align*}
By definition, the interval $[y,x]$ intersects all the intervals $I_{n+1,j'}\subset I_{n,j}$ with $j' \in \MM_{n+1}$. Since the intervals of $(n+1)$-th generation have length $|I|/b_{n+1}$ and are separated by gaps of length at least $(1-|I|)/b_{n+1}$, this implies 
\begin{equation}\label{eq:card<}
\card \MM_{n+1}(j) \leq (x-y)b_{n+1}+ 1 + |I|.
\end{equation}

By Lemma~\ref{lem:||>},
\[
r \geq |f(x)-f(y)| \geq c_1 d_n(x-y) - c_2a_{n+1},
\]
which implies
\[
x-y \leq \frac{r+c_2a_{n+1}}{c_1d_n} < c\frac{r+a_{n+1}}{d_n}
\]
for a suitable constant $c > 0$, so \eqref{eq:card<} gives
\begin{equation}\label{eq:M_n+1}
\card \MM_{n+1}(j) \leq c\frac{(r+a_{n+1})b_{n+1}}{d_n}+ 1 + |I|.
\end{equation}

Suppose $n < l$ for $l$ from \eqref{eq:l}. Then $a_{n+1} > r$, so \eqref{eq:M_n+1} gives
\begin{equation}\label{eq:preCase1}
\card \MM_{n+1}(j) < 2c \frac{a_{n+1}b_{n+1}}{d_n}+ 1 + |I|.
\end{equation}
Let $c_3 = 2c  + 1 + |I|$. If $a_{n+1} b_{n+1} < d_n$, then by \eqref{eq:preCase1},
\begin{equation}\label{eq:case1a}
\card \MM_{n+1}(j) \leq c_3 < c_3\frac{d_{n+1}}{d_n}.
\end{equation}
Otherwise, if $a_{n+1} b_{n+1} \geq d_n$, then by \eqref{eq:preCase1},
\begin{equation}\label{eq:case1b}
\card \MM_{n+1}(j) \leq c_3 \frac{a_{n+1}b_{n+1}}{d_n} < c_3 \frac{d_{n+1}}{d_n}.
\end{equation}

Suppose now $n \geq l$. Then $a_{n+1} \leq r$, so \eqref{eq:M_n+1} gives
\begin{equation}\label{eq:preCase2}
\card \MM_{n+1}(j) \leq 2 c \frac{rb_{n+1}}{d_n}+ 1 + |I|.
\end{equation}
If $d_n/b_{n+1} > r$, then by \eqref{eq:preCase2},
\begin{equation}\label{eq:case2}
\card \MM_{n+1}(j) < c_3.
\end{equation}
Otherwise, if $d_n/b_{n+1} \leq r$, then by \eqref{eq:preCase2},
\begin{equation}\label{eq:case3}
\card \MM_{n+1}(j) \leq c_3\frac{rb_{n+1}}{d_n}.
\end{equation}
Using \eqref{eq:M(j)}, \eqref{eq:case1a}, \eqref{eq:case1b}, \eqref{eq:case2} and \eqref{eq:case3}, we end the proof of the lemma.
\end{proof}

\begin{proof}[Proof of Proposition~\rm\ref{prop:HD}]
Consider $k = k(r)$, $m = m(r)$, $l = l(r)$ from \eqref{eq:k,m} and \eqref{eq:l}. In view of Lemma~\ref{lem:liminf}, it is sufficient to show
\begin{equation}\label{eq:nu>}
\frac{\log \nu(Q_r(t))}{\log r} \geq 1 +
\max\left(\frac{\log d_k}{\log (b_{k+1}/m)} - o(1),\, \frac{\log (d_{k+1}/m)}{\log (b_{k+1}/m)} - o(1)\right)
\end{equation}
as $r \to 0^+$. First we show that
\begin{equation}\label{eq:>I}
\frac{\log \nu(Q_r(t))}{\log r} \geq 1 + \frac{\log d_k}{\log (b_{k+1}/m)} - o(1)
\end{equation}
as $r \to 0^+$. Let
\[
l_1 = \max\{k,\, l\}, \qquad s_1 = \min\{n \geq l_1: d_n/b_{n+1} \leq r\}.
\]
By Lemma~\ref{lem:d/b}, the number $s_1$ is well-defined. By \eqref{eq:M_k} and Lemma~\ref{lem:M/M}, we have
\[
0 < \card\MM_{s_1} \leq 2 c_3^{s_1- k} \frac{d_{l_1}}{d_k}
\]
and
\[
0 < \card\MM_{s_1 + 1} \leq 2 c_3^{s_1 - k+1} \frac{d_{l_1}rb_{s_1+1}}{d_kd_{s_1}}\leq 2 c_3^{s_1 - k+1} \frac{r b_{s_1+1}}{d_k},
\]
so by \eqref{eq:mu(I)} and Lemma~\ref{lem:d/b}, 
\begin{equation}\label{eq:nu<I}
0 < \nu(Q_r(t)) \leq \mu\left(\bigcup_{j \in \MM_{s_1}} I_{s_1, j}\right) < \frac{\card\MM_{s_1}}{q^{s_1} b_{s_1}} \leq 
\frac{2 c_3^{s_1- k}}{q^{s_1}} \frac{d_{l_1}}{d_kb_{s_1}} < C^{s_1} \frac{b_{l_1}}{b_{s_1}} \leq \frac{C^{s_1}}{B_{l_1}^{s_1-l_1}}
\end{equation}
and
\begin{equation}\label{eq:nu<II}
0 < \nu(Q_r(t)) \leq \mu\left(\bigcup_{j \in \MM_{s_1+1}} I_{s_1+1, j}\right) < \frac{\card\MM_{s_1+1}}{q^{s_1+1} b_{s_1+1}}\leq \frac{2 c_3^{s_1 - k+1}}{q^{s_1+1}} \frac{r}{d_k} < C^{s_1}\frac{r}{d_k},
\end{equation}
where
\[
B_n = \inf_{i\geq n}\frac{b_{i+1}}{b_i}
\]
and $C > 0$ is a suitable constant. (Since $b_{n+1}/b_n \to \infty$, we have $B_{n+1} \geq B_n$ and $B_n \to \infty$.) If
\[
\frac{C^{s_1}}{B_{l_1}^{s_1-l_1}} \leq r^2,
\]
then by \eqref{eq:nu<I}, Lemma~\ref{lem:d/b} and \eqref{eq:k,m}, 
\[
\frac{\log \nu(Q_r(t))}{\log r} \geq 2 > 1 + \frac{\log d_k}{\log b_k} > 1 + \frac{\log d_k}{\log (b_{k+1}/m)},
\]
so \eqref{eq:>I} is satisfied. Hence, we can assume
\[
\frac{C^{s_1}}{B_{l_1}^{s_1-l_1}} > r^2,
\]
which implies
\begin{equation}\label{eq:B<}
(s_1 - l_1)\log B_{l_1} - s_1 \log C < -2 \log r.
\end{equation}
By \eqref{eq:k,m} and \eqref{eq:nu<II},
\[
\frac{\log \nu(Q_r(t))}{\log r} \geq  1 + \frac{\log d_k}{\log (b_{k+1}/m)} - \frac{s_1 \log C}{-\log r}.
\]
Now, if $s_1 \leq 2l_1$, then by \eqref{eq:n/b_n} and Lemma~\ref{lem:l/n}, we have \eqref{eq:>I}. Otherwise, if $s_1 > 2l_1$, then by \eqref{eq:B<}, 
\[
\frac{s_1\log C}{-\log r} < \frac{4\log C}{\log B_{l_1} - 2\log C} \to 0
\]
as $r \to 0^+$, so \eqref{eq:>I} holds.

Now we show that
\begin{equation}\label{eq:>II}
\frac{\log \nu(Q_r(t))}{\log r} \geq 1 + \frac{\log (d_{k+1}/m)}{\log (b_{k+1}/m)} - o(1)
\end{equation}
as $r \to 0^+$. The proof is analogous to the proof of \eqref{eq:>I}. Let
\[
l_2 = \max\{k+1, \, l\}, \qquad s_2 = \min\left\{n \geq l_2: d_n/b_{n+1} \leq r\right\}.
\]
By Lemma~\ref{lem:d/b}, the number $s_2$ is well-defined. By \eqref{eq:M_k} and Lemma~\ref{lem:M/M}, we have
\[
\card\MM_{s_2} \leq c_3^{s_2- k-1} (m+2)\frac{d_{l_2}}{d_{k+1}}
\]
and
\[
\card\MM_{s_2 + 1} \leq c_3^{s_2 - k}(m+2) \frac{d_{l_2}rb_{s_2+1}}{d_{k+1}d_{s_2}}\leq c_3^{s_2 - k}(m+2) \frac{r b_{s_2+1}}{d_{k+1}},
\]
so by \eqref{eq:mu(I)} and Lemma~\ref{lem:d/b}, 
\begin{equation}\label{eq:nu<I'}
\nu(Q_r(t)) \leq \mu\left(\bigcup_{j \in \MM_{s_2}} I_{s_2, j}\right) < \frac{\card\MM_{s_2}}{q^{s_2} b_{s_2}}\leq \frac{c_3^{s_2- k-1}(m+2)}{q^{s_2}} \frac{d_{l_2}}{d_{k+1} b_{s_2}} < C^{s_2} \frac{mb_{l_2}}{b_{s_2}} \leq C^{s_2}\frac{m}{B_{l_2}^{s_2-l_2}}
\end{equation}
and
\begin{equation}\label{eq:nu<II'}
\nu(Q_r(t)) \leq \mu\left(\bigcup_{j \in \MM_{s_2+1}} I_{s_2+1, j}\right) < \frac{\card\MM_{s_2+1}}{q^{s_2+1} b_{s_2+1}}\leq\frac{c_3^{s_2 - k} (m+2)}{q^{s_2+1}} \frac{r}{d_{k+1}} < C^{s_2}\frac{m r}{d_{k+1}},
\end{equation}
for a suitable constant $C > 0$. If
\[
\frac{C^{s_2}}{B_{l_2}^{s_2-l_2}} \leq r^2,
\]
then by \eqref{eq:nu<I'}, Lemma~\ref{lem:d/b} and \eqref{eq:k,m}, 
\[
\frac{\log \nu(Q_r(t))}{\log r} \geq 2 > 1 + \frac{\log (d_{k+1}/m)}{\log (b_{k+1}/m)},
\]
so \eqref{eq:>II} holds. Hence, we can assume
\[
\frac{C^{s_2}}{B_{l_2}^{s_2-l_2}} > r^2,
\]
which implies
\begin{equation}\label{eq:B<'}
(s_2 - l_2)\log B_{l_2} - s_2 \log C < -2 \log r.
\end{equation}
By \eqref{eq:k,m} and \eqref{eq:nu<II'},
\[
\frac{\log \nu(Q_r(t))}{\log r} \geq  1 + \frac{\log (d_{k+1}/m)}{\log (b_{k+1}/m)} - \frac{s_2 \log C}{-\log r}.
\]
If $s_2 \leq 2 l_2$, then by \eqref{eq:n/b_n} and Lemma~\ref{lem:l/n}, we have \eqref{eq:>II}. Otherwise, if $s_2 > 2l_2$, then by \eqref{eq:B<'}, 
\[
\frac{s_2\log C}{-\log r} < \frac{4\log C}{\log B_{l_2} - 2\log C} \to 0
\]
as $r \to 0^+$, so \eqref{eq:>II} is satisfied.

In this way we showed \eqref{eq:>I} and \eqref{eq:>II}, which implies \eqref{eq:nu>} and ends the proof of the proposition.
\end{proof}

Proposition~\ref{prop:HD} together with Proposition~\ref{prop:holder}~(c) and Lemma~\ref{lem:frostman} give immediately the following corollary.

\begin{cor}\label{cor:HD}
\[
\dim_H(\graph f|_{\II}) \geq  1 + \liminf_{n\to\infty} \frac{\log^+ d_n}{\log (b_{n+1}d_n/d_{n+1})}.
\]
\qed
\end{cor}

Propositions~\ref{prop:upperBD}--\ref{prop:lowerBD} and Corollary~\ref{cor:HD} end the proof of Theorem~\ref{thm:main}.


\begin{thebibliography}{99}

\bibitem{BU} A.~S.~Besicovitch and H.~D.~Ursell, \emph{Sets of fractional dimensions (V): On dimensional numbers of some continuous curves}, J. London Math. Soc. 12 (1937), 18--25.

\bibitem{C} A.~Carvalho, \emph{Hausdorff dimension of scale-sparse Weierstrass-type functions}, Fund. Math. 213 (2011), 1--13.

\bibitem{F} K.~J.~Falconer, \emph{Fractal Geometry: Mathematical
Foundations and Applications}, J.~Wiley \& Sons, 1990.

\bibitem{He} Y.~Heurteaux, \emph{Weierstrass functions with random phases}, Trans. Amer. Math. Soc. 335 (2003), 3065--3077.

\bibitem{Hu} B.~R.~Hunt, \emph{Hausdorff dimension of graphs of
Weierstrass functions}, Proc. Amer. Math. Soc. 126 (1998), 791--800.

\bibitem{L} F.~Ledrappier, \emph{On the dimension of some graphs},
Contemp. Math. 135 (1992), 285--293.

\bibitem{Li} Y.-Y.~Liu, \emph{ A function whose graph is of dimension $1$ and has locally an infinite one-dimensional Hausdorff measure}, C. R. Acad. Sci. Paris Sér. I Math. 332 (2001), 19--23.

\bibitem{M} P.~Mattila, \emph{Geometry of sets and measures in Euclidean spaces}, Cambridge Univ. Press, Cambridge, 1995.

\bibitem{MW} R.~D.~Mauldin and S.~C.~Williams, \emph{On the
Hausdorff dimension of some graphs}, Trans. Amer. Math. Soc. 298
(1986), 793--804.

\bibitem{PU} F.~Przytycki and M.~Urba\'{n}ski, \emph{On the
Hausdorff dimension of some fractal sets}, Studia Math. 93 (1989),
155--186.

\bibitem{W} P.~Wingren, \emph{Concerning a real-valued continuous function on the interval with graph of Hausdorff dimension $2$}, L'Enseign. Math. 41 (1995) 103--110. 

\end{thebibliography}
\end{document}